\documentclass[aos]{imsart}

\RequirePackage{amsthm,amsmath,amsfonts,amssymb}
\RequirePackage[numbers]{natbib}
\RequirePackage[colorlinks,citecolor=blue,urlcolor=blue]{hyperref}%% uncomment this for coloring bibliography citations and linked URLs
\RequirePackage{graphicx}%% uncomment this for including figures

\startlocaldefs
\theoremstyle{plain}

\newtheorem{remark}{Remark}
\newtheorem{corollary}{Corollary}[section]

\newtheorem{theorem}{Theorem}[section]
\newtheorem{lemma}{Lemma}[section]
\theoremstyle{remark}

\endlocaldefs

\begin{document}

\begin{frontmatter}
\title{FWER Goes to Zero for Correlated Normal}
%\title{A sample article title with some additional note\thanksref{T1}}
\runtitle{FWER Goes to Zero for Correlated Normal}
%\thankstext{T1}{A sample of additional note to the title.}

\begin{aug}
%%%%%%%%%%%%%%%%%%%%%%%%%%%%%%%%%%%%%%%%%%%%%%%
%% Only one address is permitted per author. %%
%% Only division, organization and e-mail is %%
%% included in the address.                  %%
%% Additional information can be included in %%
%% the Acknowledgments section if necessary. %%
%%%%%%%%%%%%%%%%%%%%%%%%%%%%%%%%%%%%%%%%%%%%%%%
\author[A]{\fnms{Monitirtha} \snm{Dey}\ead[label=e1]{monitirtha.d\_r@isical.ac.in}}
\and
\author[A]{\fnms{Subir} \snm{Kumar Bhandari}\ead[label=e2,mark]{subir@isical.ac.in}}

%%%%%%%%%%%%%%%%%%%%%%%%%%%%%%%%%%%%%%%%%%%%%%
%% Addresses                                %%
%%%%%%%%%%%%%%%%%%%%%%%%%%%%%%%%%%%%%%%%%%%%%%
\address[A]{Indian Statistical Institute, 
\printead{e1}, \printead{e2}}

\end{aug}

\begin{abstract}
Familywise error rate (FWER) has been a cornerstone in simultaneous inference for decades, and the classical Bonferroni method has been one of the most prominent frequentist approaches for controlling FWER. The present article studies the limiting behavior of Bonferroni FWER in a multiple testing problem as the number of hypotheses grows to infinity. We establish that in the equicorrelated normal setup with positive equicorrelation, Bonferroni FWER tends to zero asymptotically. We extend this result for generalized familywise error rates and to arbitrarily correlated setups.
\end{abstract}
 


\end{frontmatter}
%%%%%%%%%%%%%%%%%%%%%%%%%%%%%%%%%%%%%%%%%%%%%%
%% Please use \tableofcontents for articles %%
%% with 50 pages and more                   %%
%%%%%%%%%%%%%%%%%%%%%%%%%%%%%%%%%%%%%%%%%%%%%%
%\tableofcontents

\section{Introduction}

%%%%%%%%%%%%%%%%%%%%%%%%%%%%%%%%%%%%%%%%%%%%%%
%% `\ ' is used here because TeX ignores    %%
%% spaces after text commands.              %%
%%%%%%%%%%%%%%%%%%%%%%%%%%%%%%%%%%%%%%%%%%%%%%
Large-scale multiple testing problems in various disciplines routinely involve dependent observations. For example, in genome-wide association studies, high-density SNP (single nucleotide polymorphism) markers used to analyze genetic diversity exhibit high correlation. In spatial data with close geographical locations, the test statistics corresponding to different hypotheses often get influenced by each other. Multistage clinical trials and functional magnetic resonance imaging studies also concern variables with complex and unknown dependence structures. However, most classical multiple testing procedures controlling false discovery rate (FDR) or familywise error rate (FWER) typically ignore this dependence. As a result, the problem of capturing the association among observations and extending the existing classical methods under association has attracted considerable attention in recent times. Multiple testing procedures under dependence have been discussed by Sun and Cai~\cite{r11}, Efron~\cite{r3}, Liu, Zhang and Page~\cite{r8} among others. Efron~\cite{r4} remarks that the correlation penalty on the summary statistics depends on the root mean square (RMS) of correlations. Efron~\cite{r5} contains an excellent review of the relevant literature.

This work focuses on the FWER, one of the most widely considered frequentist approaches to multiple testing. Defined as the probability of erroneously rejecting at least one true null hypothesis, it is the most natural generalization of type I error rate to simultaneous testing problems. Controlling FWER at a target level $\alpha$ requires conducting each of the individual hypothesis tests at lower levels. For instance, the Bonferroni method divides $\alpha$ by the number of tests considered.

We have considered the equicorrelated normal distribution with non-negative correlation $\rho$. Das and Bhandari~\cite{r1} have shown that under this setup, FWER($\rho$) is a convex function in correlation $\rho$ as the number of hypotheses grows to infinity. Consequently, they prove that the FWER of the Bonferroni procedure is bounded by $\alpha(1-\rho)$ where $\alpha$ is the target level. Dey~\cite{r2} has obtained upper bounds on FWER of the Bonferroni method in the equicorrelated non-asymptotic setup. Here we establish that the Bonferroni FWER($\rho$) tends to zero asymptotically for any positive $\rho$.

The approximation on FWER proposed in this work provides an estimate on the c.d.f. of the failure time of the parallel systems. Loperfido~\cite{r9} has found that the maximum of $n$ observations from equicorrelated normal distribution follows ($n-1$) dimensional skew-normal distribution. While the c.d.f. of a multivariate skew-normal distribution is very difficult to work with, an asymptotic approximation on the c.d.f. may be obtained along similar lines as in this article.

The rest of this paper is organized as follows. In Section 2, we introduce the necessary notation and set up the framework formally. Section 3 contains theoretical results about the limiting behavior of the FWER in equicorrelated and arbitrarily correlated normal setups. Section 4 presents an extension of our main contribution to generalized familywise error rates. Section 5 includes simulation findings that empirically demonstrate our results. We conclude with a brief discussion in Section 6.

\section{Preliminaries} We consider an \textit{Gaussian sequence model}:
$$X_{i} {\sim} \mathcal{N}(\mu_{i},1),  i=1, \ldots,n,$$
where $X_{i}$’s are independent and we are interested in the $n$ null hypotheses $H_{0i}:\mu_{i}=0$. The global null $ H_{0}=\bigcap_{i=1}^{n} H_{0 i}$ asserts that all means $\mu_i = 0$, while under the alternative, some mean is non-zero. We have assumed equicorrelated setup, i.e, $\operatorname{\mathbb{C}orr}\left(X_{i}, X_{j}\right)=\rho$ for $i \neq j$ where $\rho \in [0,1]$. Throughout this work, $\phi$ and $\Phi$ will be used to denote the probability density function and cumulative density function of standard normal distribution respectively.

We consider FWER of Bonferroni’s method which sets a same cut-off for each of the $n$ hypotheses. In the one-sided setting, it rejects $H_{0i}$ if $X_{i}>\Phi ^{-1}(1-\alpha/n) (=c_{\alpha,n}, \text{say})$ where $\alpha \in (0,1)$ is the target level at which one wishes to control FWER. In other words, the FWER of Bonferroni’s method under the equicorrelated setup  is given by
\begin{equation}
    FWER(n, \alpha,\rho) =\mathbb{P}\left(X_{i}>c_{\alpha,n}\right. \text{for some} \left.i \mid H_{0}\right)
=\mathbb{P}_{H_{0}}\bigg(\bigcup_{i=1}^{n}\{X_i > c_{\alpha,n}\}\bigg). \label{defFWER}
\end{equation}
Das and Bhandari~\cite{r1} consider the above equicorrelated framework and show the following:
\begin{theorem}\label{thm2.1}
Suppose each $H_{0i}$ is being tested at size $\alpha_{n}$. If $\displaystyle \lim _{n \rightarrow \infty} n \alpha_{n}=\alpha \in(0,1)$ then, $FWER$ asymptotically is a convex function in $\rho \in [0,1]$.
\end{theorem}
For Bonferroni’s procedure, $\alpha_n = \alpha/n$ and thus Theorem \ref{thm2.1} also applies for Bonferroni’s method. Moreover, the following corollary is established using Theorem \ref{thm2.1} in~\cite{r1}:
\begin{corollary}\label{c1}
Given any $\alpha \in (0,1)$ and $\rho \in [0,1]$, $FWER(n, \alpha, \rho)$ is asymptotically bounded by $\alpha(1-\rho)$.
\end{corollary}
To prove Theorem \ref{thm2.1}, Das and Bhandari~\cite{r1} consider the function $H_{n}(\rho)=1-FWER(n, \alpha, \rho)$ and study its behaviour.
Evidently the sequence $\left\{X_{r}\right\}_{r \geq 1}$ is exchangeable under $H_{0}$ for the equicorrelated set-up. In other words,  $$\left(X_{i_{1}}, \ldots, X_{i_{k}}\right) \sim N_{k}\left(\mathbf{0}_{\mathbf{k}},(1-\rho) I_{k}+\rho J_{k}\right))$$ 
where $J_{k}$ is the $k \times k$ matrix of all ones. Then, for each $i \geq 1$, $X_{i}=\theta+Z_{i}$
where $\theta$ is a normal random variable having mean $0$, independent of $\left\{Z_{n}\right\}_{n \geq 1}$ and $Z_{i}$’s are i.i.d normal random variables. $\operatorname{Cov}\left(X_{i}, X_{j}\right)=\rho$ gives $\operatorname{Var}(\theta)=\rho$. This implies
$\theta \sim \mathcal{N}(0, \rho)$ and $Z_{i} \overset{iid}{\sim} \mathcal{N}(0,1-\rho)$ for each $i \geq 1$. Thus,
\begin{eqnarray}
H_{n}(\rho) & = & \mathbb{P}\left(\theta+Z_{i} \leq c_{\alpha,n} \quad \forall i=1,2, \ldots, n \mid H_{0}\right) \nonumber\\
            & = & \mathbb{E}_{\theta}\left[\Phi\left(\frac{c_{\alpha,n} -\theta}{\sqrt{1-\rho}}\right)^{n}\right]\\
            & = & \mathbb{E}\left[\Phi\left(\frac{c_{\alpha,n} +\sqrt{\rho} Z}{\sqrt{1-\rho}}\right)^{n}\right] \quad \text{(where $Z \sim N(0,1)$)}.\nonumber
\end{eqnarray}
We study the limiting behavior of $H_{n}(\rho)$ as $n \to \infty$ in the next section.
\section{Main Results} Firstly, we state the foremost theoretical result of this work.
\begin{theorem}\label{thm3.1}
Given any $\alpha \in (0,1)$ and $\rho \in (0,1]$, $\lim_{n \to \infty} FWER(n, \alpha, \rho) = 0.$
\end{theorem}
Theorem \ref{thm3.1}, a much stronger result of Corollary \ref{c1}, highlights the fundamental problem of Bonferroni method as a multiple testing procedure. Note that in proving Theorem \ref{thm3.1}, it suffices to show $\lim_{n \to \infty}H_{n}(\rho) =1$. This will be established in the following steps:
\begin{enumerate}
    \item Finding an approximation for $c_{\alpha,n}$ for large $n$.
    \item With the help of the approximation, showing for each $\alpha$ and $\rho$ in $(0,1)$,
$$\lim_{n \to \infty} \Bigg[\Phi\bigg(\frac{c_{\alpha,n}}{\sqrt{1-\rho}} \bigg)\Bigg]^{n} = 1.$$
    \item Showing that for any fixed real number $t$ and for each $\alpha$, $\rho$ in $(0,1)$, we have
$$\lim_{n \to \infty} \Bigg[\Phi\bigg(\frac{c_{\alpha, n}+t}{\sqrt{1-\rho}}\bigg)\Bigg]^{n}= 1.$$ Then, applying dominated convergence theorem to show $H_{n}(\rho) \longrightarrow1$ as $n \to \infty$.
\end{enumerate}
We explicate the steps in three lemmas.

\begin{lemma}\label{l1}
Given any $\alpha \in (0,1)$, $c_{\alpha,n} \leq \sqrt{2 \ln(n)}$ for all sufficiently large $n$. Also, $\frac{c_{\alpha,n}} {\sqrt{2 \ln(n)}} \longrightarrow1$ as $n \to \infty$.
\end{lemma}
The lemma follows from utilizing the following well-known result by Gordon~\cite{r6} once one replaces $x$ by $c_{\alpha,n}$ and observes that $\Phi(c_{\alpha,n})=1-\alpha/n$.

\begin{theorem}[Gordon]\label{thm3.2}
For arbitrary positive number $x>0$, the inequalities
$$
\frac{x \phi(x)}{1+x^{2}} <1-\Phi(x)<\frac{\phi(x)}{x}
$$
hold. In particular, 
$$\lim_{x \to \infty}\frac{x(1-\Phi(x))}{\phi(x)} = 1.$$
\end{theorem}

\begin{lemma}\label{l2}
For each $\alpha$ and $\rho$ in $(0,1)$, 
\begin{equation}
    \lim_{n \to \infty} \Bigg[\Phi\bigg(\frac{c_{\alpha,n}}{\sqrt{1-\rho}} \bigg)\Bigg]^{n} = 1.
\end{equation}
\end{lemma}

\begin{proof}[Proof of Lemma \ref{l2}] Observe that

\begin{eqnarray}
\lim_{n \to \infty} \Bigg[\Phi\bigg(\frac{\sqrt{2 \ln n}}{\sqrt{1-\rho}}\bigg) \Bigg]^{n}& = & \lim_{n \to \infty} \Bigg[\Phi\bigg(\sqrt{2 \ln \left(n^{\frac{1}{1-\rho}} \right)}\bigg)\Bigg]^{n} \nonumber \\
    & = & \lim_{n \to \infty}  \Bigg[\Bigg[\Phi\bigg(\sqrt{2 \ln \left(n^{\frac{1}{1-\rho}} \right)}\bigg)\Bigg]^{n^{\frac{1}{1-\rho}}}\Bigg] ^{n^{\frac{-\rho}{1-\rho}}} \nonumber \\
    & = & \lim_{n \to \infty} \Bigg[ \bigg[\Phi(\sqrt{2 \ln m})\bigg]^{m}\Bigg]^{n^{\frac{-\rho}{1-\rho}}} \quad \text{(writing $m=n^{\frac{1}{1-\rho}}$).} \label{4}
\end{eqnarray}
Invoking Lemma \ref{l1}, we obtain from \eqref{4}
\begin{eqnarray}
    \lim_{n \to \infty} \Bigg[\Phi\bigg(\frac{\sqrt{2 \ln n}}{\sqrt{1-\rho}}\bigg) \Bigg]^{n} & \geq & \lim_{n \to \infty} \bigg[ \big[\Phi(c_{\alpha, m})\big]^{m}\bigg]^{n^{\frac{-\rho}{1-\rho}}} \nonumber \\
    & = & \lim_{n \to \infty} \big[H_{m}(0)\big]^{n^{\frac{-\rho}{1-\rho}}} \quad \text{(from definition of $H_{n}(\rho)$)} \nonumber \\
    & = & \lim_{n \to \infty} \big[e^{-\alpha}\big]^{n^{\frac{-\rho}{1-\rho}}},
\end{eqnarray}
where the last step emanates from the fact that, for $\rho=0$, $X_{i}$'s become independent and so $H_{n}(0)=\displaystyle \prod_{i=1}^n \mathbb{P}_{H_{0}}(X_i \leq c)=\left(1-\frac{\alpha}{n}\right)^n \to  e^{-\alpha}$ as $n \to \infty$. 

\noindent For $\rho=0$, $\displaystyle \lim_{n \to \infty} \big[e^{-\alpha}\big]^{n^{\frac{-\rho}{1-\rho}}} = e^{-\alpha}$ and for $\rho \in (0,1)$, $\displaystyle \lim_{n \to \infty} \big[e^{-\alpha}\big]^{n^{\frac{-\rho}{1-\rho}}} = 1$. Therefore, 
\begin{equation}
    \lim_{n \to \infty} \Bigg[\Phi\bigg(\frac{\sqrt{2 \ln n}}{\sqrt{1-\rho}}\bigg)\Bigg]^{n} = 1 \quad \text{for each }\rho \in (0,1). \label{e6}
\end{equation}

\noindent Lemma \ref{l1} enables us to choose $\rho_{1} \in (0, \rho]$ such that 
$$ \Bigg[\Phi\bigg(\frac{c_{\alpha, n}}{\sqrt{1-\rho}}\bigg)\Bigg]^{n} \geq \Bigg[\Phi\bigg(\frac{\sqrt{2 \ln n}}{\sqrt{1-\rho_{1}}}\bigg)\Bigg]^{n} \quad \text{for all sufficiently large $n$}.$$
The rest follows from \eqref{e6}.
\end{proof}
We shall now prove the following generalization of Lemma \ref{l2}:
\begin{lemma}\label{l3}
For any fixed real number $t$, we have
\begin{equation}
    \lim_{n \to \infty} \Bigg[\Phi\bigg(\frac{c_{\alpha, n}+t}{\sqrt{1-\rho}}\bigg) \Bigg]^{n}= 1 \quad \forall \alpha \in (0,1), \forall \rho \in (0,1). \label{e7}
\end{equation}
\end{lemma}
\begin{proof}[Proof of Lemma \ref{l3}] For positive values of $t$, the result is immediate from Lemma \ref{l2} using the increasing property of $\Phi(\cdot)$. For negative values of $t$, we choose $\rho_{2} \in (0, \rho]$ such that 
$$ \Bigg[\Phi\bigg(\frac{c_{\alpha, n}+t}{\sqrt{1-\rho}}\bigg)\Bigg]^{n} \geq \Bigg[\Phi\bigg(\frac{c_{\alpha, n}}{\sqrt{1-\rho_2}}\bigg)\Bigg]^{n} \quad \text{for all sufficiently large $n$}.$$
Such a $\rho_{2}$ can always be chosen because $t$ is fixed and $c_{\alpha, n} \longrightarrow \infty$ as $n \to \infty$ for any $\alpha \in (0,1)$. An application of Lemma \ref{l2} gives us \eqref{e7}.
\end{proof}
We are now in a position to establish Theorem \ref{thm3.1}.

\begin{proof}[Proof of Theorem \ref{thm3.1}] Let $Z \sim N(0,1)$ and $\Phi^{n}(\cdot)$ denote $[\Phi(\cdot)]^{n}$.
We obtain from \eqref{e6} and \eqref{e7},
$$\lim_{n \to \infty} \Bigg[ \Phi^{n}\left(\frac{c_{\alpha,n} +\sqrt{\rho} Z}{\sqrt{1-\rho}}\right) - \Phi^{n}\bigg(\frac{c_{\alpha, n}}{\sqrt{1-\rho}}\bigg) \Bigg]= 0 \quad \text{almost everywhere in $Z$},$$
for each $\alpha$ and $\rho$ in  $(0,1)$. Moreover, it is easily seen that, for each $\alpha$ and $\rho$ in  $(0,1)$, 
$$\left|  \Phi^{n}\left(\frac{c_{\alpha,n} +\sqrt{\rho} Z}{\sqrt{1-\rho}}\right) - \Phi^{n}\bigg(\frac{c_{\alpha, n}}{\sqrt{1-\rho}}\bigg) \right| \leq 1.$$
Therefore, an application of dominated convergence theorem would yield,
$$\lim_{n \to \infty} \Bigg[ \mathbb{E}\bigg[\Phi^{n}\left(\frac{c_{\alpha,n} +\sqrt{\rho} Z}{\sqrt{1-\rho}}\right)\bigg] - \mathbb{E}\bigg[\Phi^{n}\bigg(\frac{c_{\alpha, n}}{\sqrt{1-\rho}}\bigg)\bigg] \Bigg]= 0.$$
Applying Lemma \ref{l2}, we have $\lim_{n \to \infty} H_{n}(\rho) = 1$ for each $\rho \in (0,1)$. Definition of $H_{n}(\rho)$ gives
$$\lim_{n \to \infty} FWER(n, \alpha, \rho) = 0 \quad \forall \rho \in (0,1).$$
Thus, the only thing remaining to show is $\lim_{n \to \infty} FWER(n, \alpha, 1) = 0.$ For $\rho=1$, $X_{i}=X_{j}$ almost surely $\forall i \neq j$. Consequently, one rejection would imply rejection of all null hypotheses and $H_{n}(\rho)= \mathbb{P}(X_{1} \leq c_{\alpha,n}) = 1- \alpha/n$. Therefore, in this case also, $H_{n}(\rho)$ tends to $1$ as $n \to \infty$, completing the proof.
\end{proof}
\begin{remark}
Proofs of Lemma \ref{l2} and Theorem \ref{thm3.1} provide us with the following approximation to FWER of Bonferroni procedure for large $n$:
$$FWER(n, \alpha, \rho) \approx 1 - e^{-\alpha \cdot n^{\frac{-\rho}{1-\rho}}}.$$
\end{remark}
\begin{remark}
We have seen in the proof of Lemma \ref{l2} that  $H_{n}(0) \longrightarrow e^{-\alpha}$ as $n \to \infty$. Thus, $FWER(n, \alpha, 0) \to 1 - e^{-\alpha}$ as $n \to \infty$. Combining this with Theorem \ref{thm3.1}, we obtain that the limiting FWER is a convex function, though discontinuous at $0$. Thus, Theorem \ref{thm3.1} does not contradict Theorem \ref{thm2.1} established by Das and Bhandari~\cite{r1}, rather implies it.
\end{remark}

We have considered an equicorrelated normal setup so far. However, problems involving variables with a more general dependence structure need to be addressed with more general correlation matrices. Hence, the study of the limiting behavior of FWER in arbitrarily correlated normal setups becomes crucial. Towards this, we consider the same \textit{Gaussian sequence model} as in Section 2, but now we assume $\operatorname{\mathbb{C}orr}\left(X_{i}, X_{j}\right)=\rho_{ij}$ for $i \neq j$ where $\rho_{ij} \in (0,1]$. Let $\mathbf{\Sigma}_{n}$ be the correlation matrix of $X_{1}, \ldots, X_{n}$ and $FWER(n, \alpha, \mathbf{\Sigma}_{n})$ denote the FWER of Bonferroni’s method under this setup. The following result is a generalization of Theorem \ref{thm3.1} for this setup.

\begin{theorem}\label{thm3.3}
Suppose $\liminf \rho_{ij}=\delta>0$. Then, for any $\alpha \in (0,1)$,
$$\lim_{n \to \infty}FWER(n,\alpha,\mathbf{\Sigma}_{n}) = 0.$$
\end{theorem}
We establish this using the famous inequality due to Slepian~\cite{r10}:
\begin{theorem}[Slepian]\label{thm3.4}
Let $\mathbf{X}$ be distributed according to $N(\mathbf{0}, \mathbf{\Sigma})$, where $\mathbf{\Sigma}$ is a correlation matrix. For an arbitrary but fixed real 
vector $\mathbf{a}=\left(a_{1}, \ldots, a_{k}\right)^{\prime}$, consider the quadrant probability
$$g(k, \mathbf{a}, \mathbf{\Sigma})=\mathbb{P}_{\mathbf{\Sigma}}\left[\bigcap_{i=1}^{k}\left\{X_{i} \leqslant a_{i}\right\}\right].$$ Let $\mathbf{R}=\left(\rho_{i j}\right)$ and $\mathbf{T}=\left(\tau_{i j}\right)$ be two positive semidefinite correlation matrices. If $\rho_{i j} \geqslant \tau_{i j}$ holds for all $i, j$, then $g(k, \mathbf{a}, \mathbf{R}) \geq g(k, \mathbf{a}, \mathbf{T})$, i.e
$$
\mathbb{P}_{\mathbf{\Sigma}=\mathbf{R}}\left[\bigcap_{i=1}^{k}\left\{X_{i} \leqslant a_{i}\right\}\right] \geqslant \mathbb{P}_{\mathbf{\Sigma}=\mathbf{T}}\left[\bigcap_{i=1}^{k}\left\{X_{i} \leqslant a_{i}\right\}\right]
$$
holds for all $\mathbf{a}=\left(a_{1}, \ldots, a_{k}\right)^{\prime} .$ Furthermore, the inequality is strict if $\mathbf{R}, \mathbf{T}$ are positive definite and if the strict inequality $\rho_{i j}>\tau_{i j}$ holds for some $i, j$.
\end{theorem}

\begin{proof}[Proof of Theorem \ref{thm3.3}]
Note that, for any $n \times n$ correlation matrix $\mathbf{\Sigma}_{n}$, $$FWER(n,\alpha,\mathbf{\Sigma}_{n})=1-g(n, \mathbf{a}, \mathbf{\Sigma}_{n})$$
where $g$ is as defined in Theorem \ref{thm3.4} and $a_{i}=\Phi ^{-1}(1-\alpha/n)$ for $1 \leq i \leq n$. Since $\liminf \rho_{ij}=\delta>0$, we have $g(n, \mathbf{a},\Sigma_{n}) \geq g(n, \mathbf{a}, M_{n}(\delta))$ from Theorem \ref{thm3.4}, where $M_{n}(\delta)$ denotes the $n \times n$ matrix with diagonal entries equal to $1$ and off-diagonal entries equal to $\delta$. Therefore, $FWER(n,\alpha, \mathbf{\Sigma}_{n}) \leq FWER(n,\alpha, \delta)$. The rest follows from Theorem \ref{thm3.1}. 
\end{proof}

\section{Generalized familywise error rates}The increase of computational abilities have allowed for the simultaneous testing of many (thousands or millions) statistical hypothesis tests.
In many areas, e.g microarray data analysis, the number of hypotheses under consideration is quite
large. Control of the FWER in those cases is so stringent that departures from the null hypothesis have little chance of being detected. Consequently, alternative measures of error control have been proposed in the literature.

Lehmann and Romano~\cite{r7} consider the $k$-FWER, the probability of rejecting at least $k$ true null hypotheses in a multiple testing framework.  Such an error rate criterion with $k>1$ is pertinent to settings where one may tolerate one or more false rejections, provided the number of false positives is controlled. 

Consider the arbitrarily correlated setup described in Section 3. Lehmann and Romano~\cite{r7} remark that control of the $k$-FWER allows one to decrease the Bonferroni cutoff $\Phi ^{-1}(1-\alpha/n)$
to $\Phi ^{-1}(1-k\alpha/n)$, and thereby greatly increase the ability to detect false hypotheses. Thus, for their procedure, 
\begin{align*}k\text{-}FWER(n, \alpha, \mathbf{\Sigma}_{n}) &=\mathbb{P}\left(X_{i}>\Phi ^{-1}(1-k\alpha/n) \hspace{2mm} \text{for at least $k$}
\hspace{2mm}\text{$i$'s} \mid H_{0}\right) \end{align*}
where $H_{0}$ is the intersection null hypothesis.

\vspace{2mm}
We now extend Theorem \ref{thm3.3} for generalized familywise error rates.

\begin{corollary}\label{c2}
Suppose $\liminf \rho_{ij}=\delta>0$. Then, for any $\alpha \in (0,1)$ and any positive integer $k$ satisfying $k\alpha <1$,
$$\lim_{n \to \infty}k\text{-}FWER(n,\alpha,\mathbf{\Sigma}_{n}) = 0.$$
\end{corollary}

\begin{proof}[Proof of Corollary \ref{c2}]
Evidently, for any natural number $k$ satisfying $k\alpha <1$,
$$k\text{-}FWER(n, \alpha, \mathbf{\Sigma}_{n}) \leq FWER(n, \alpha^{*}=k \alpha, \mathbf{\Sigma}_{n}).$$
The rest is immediate from Theorem \ref{thm3.3}.
\end{proof}

\newpage
\section{Simulation Study}We have mentioned in Section 2 that under $H_{0}$, $X_{i}=\theta+Z_{i}$
where $\theta \sim N(0,\rho)$, independent of $\left\{Z_{n}\right\}_{n \geq 1}$ and $Z_{i} \overset{iid}{\sim} N(0,1-\rho)$. Now, \eqref{defFWER} gives us
\begin{eqnarray}
    FWER(n, \alpha,\rho) & = & \mathbb{P}_{H_{0}}\bigg(\bigcup_{i=1}^{n}\{Z_{i}+\theta > c_{\alpha,n}\}\bigg) \nonumber \\
    & = & \mathbb{P}_{H_{0}}\bigg(\max_{1 \leq i \leq n} Z_{i}> c_{\alpha,n} - \theta\bigg) \nonumber \\
    & = & \mathbb{P}_{H_{0}}\bigg(\max_{1 \leq i \leq n} W_{i}> \frac{c_{\alpha,n} - \sqrt{\rho}\beta}{\sqrt{1-\rho}}\bigg) \quad \text{(where $W_{i}=Z_{i}/\sqrt{1-\rho}$, $\theta = \sqrt{\rho}\beta$)} \nonumber\\
    & = & \mathbb{E}_{\beta}\Bigg[\mathbb{I}\bigg\{\max_{1 \leq i \leq n} W_{i}> C_{\beta,\rho}\bigg\}\Bigg] \label{simulate}
\end{eqnarray}
where $\mathbb{I}\{A\}$ denotes the indicator variable of event $A$ and $C_{\beta,\rho}$ is the quantity it is replacing. Note that $\beta \sim N(0,1)$ independent of $W_{i}\sim N(0,1)$ under $H_{0}$. 

Equation \eqref{simulate} illustrates an elegant and computationally less expensive simulation scheme of estimating FWER given $(n, \alpha, \rho)$. Firstly, we generate $1$ million independent observations from $N(0,1)$ (these are the $\beta$ variables, serving as repetitions). Given $\rho$, we compute the cutoff $C_{\beta_{i},\rho}$ for each of the simulated $\beta_{i}$’s, $1\leq i \leq 1$ million. Given $n$, we generate $n$ independent observations from $N(0,1)$ (these are the $W_{i}$’s) and compute the maximum of these $n$ observations. We note for how many $i$’s, the maximum exceeds the cutoff $C_{\beta_{i},\rho}$. This number when divided by $1$ million gives us an estimate of FWER$(n,\alpha,\rho)$.

Table~\ref{sphericcase} presents the estimates of FWER for some large values of $n$ and $\alpha=.05$. It also includes the exact values of FWER for $\rho=0$. It is mention-worthy that, though for small values of $\rho$, there is an increase in FWER from $n=100000$ to $n=1$ million, overall for each positive $\rho$, FWER values decrease as $n$ grows indefinitely. 

\section{Concluding Remarks}
It is well-known that Bonferroni’s procedure becomes very conservative for large-scale multiple testing problems under the classical i.i.d framework. However, very little literature can be found which elucidates the magnitude of the conservativeness of Bonferroni’s method in a dependent setup. Our results address this gap in a unified manner. Our methodology, heavily based on elementary real analysis, nicely applies to general dependent structures. Even more importantly, by considering $k$-FWER instead of FWER, we encompass a much broader class of error rates.

\begin{table*}
\caption{Estimates of FWER($n,\alpha=.05,\rho$)}
\label{sphericcase}
\begin{tabular}{@{}lrrrrc@{}}
\hline
Correlation & \multicolumn{5}{c}{Number of hypotheses ($n$)}\\
\cline{2-6}
\hspace{4.7mm} $(\rho)$ & \multicolumn{1}{c}{$10000$}
& \multicolumn{1}{c}{$100000$} & \multicolumn{1}{c}{$1$ Million}
& \multicolumn{1}{c}{$10$ Million} & \multicolumn{1}{c}{$100$ Million} \\
\hline
0    & .04877069 & .04877059  & .04877058  &  .04877058 & .04877058 \\
0.1  & .007621 & .006174  & .026960  & .013166  & .001011 \\
0.2  & .014523 & .010875  & .024589  & .013381  & .002210 \\
0.3  & .014317 & .009928  & .016512  & .009003  & .001898 \\
0.4  & .011448 & .007309  & .009828  & .005118  & .001239 \\
0.5  & .008180 & .004735  & .005239  & .002523  & .000622 \\
0.6  & .005227 & .002701  & .002436  & .001110  & .000257 \\
0.7  & .002909 & .001324  & .000956  & .000382  & .000079 \\
0.8  & .001325 & .000479  & .000264  & .000081  & .000018 \\
0.9  & .000390 & .000098  & .000035  & .000013  & .000004 \\
\hline
\end{tabular}
\end{table*}

\end{document}